\documentclass[11pt,a4paper,leqno,notitlepage]{amsart}
\usepackage[latin1]{inputenc}
\usepackage[english,frenchb]{babel}
\usepackage{lmodern}
\usepackage{amsmath}
\usepackage{amsfonts}
\usepackage{amssymb}
\theoremstyle{plain}
\usepackage{amsthm}
\usepackage{graphicx}
\usepackage{mathrsfs}
\usepackage{pxfonts}
\usepackage{datetime}
\usepackage{dsfont}
\usepackage[left=3cm,right=3cm,top=3cm,bottom=3cm]{geometry}
\newtheorem{theoreme}{Theorem}[section]
\usepackage[cyr]{aeguill}
\usepackage{fancyheadings}
\usepackage{mathabx}
 \usepackage{comment}
\input xy
\xyoption{all}
\newtheorem{lemme}[theoreme]{Lemma}
\newtheorem{proposition}[theoreme]{Proposition}
\newtheorem{definition}[theoreme]{Definition}

\newtheorem{remarque}[theoreme]{Remark}
\newtheorem*{remarque*}{Remark}

\newtheorem{corollaire}[theoreme]{Corollary}

\author{}
 \def\dar[#1]{\ar@<2pt>[#1]\ar@<-2pt>[#1]}
  \def\commutatif{\ar@{}[rd]|{\circlearrowleft}}
\longdate
\title{Existence of slices on a tame context \footnote{Submitted in CEJM} 
}
\author{Sophie Marques}
\date{\today}
\newcommand{\eq}[1][r]
   {\ar@<-3pt>@{-}[#1]
    \ar@<-1pt>@{}[#1]|<{}="gauche"
    \ar@<+0pt>@{}[#1]|-{}="milieu"
    \ar@<+1pt>@{}[#1]|>{}="droite"
  \ar@/^2pt/@{-}"gauche";"milieu"
   \ar@/_2pt/@{-}"milieu";"droite"
}
\address[MARQUES Sophie]{Visiting assistant professor at New York University, Courant institute of mathematical sciences, 251 Mercer St, New York, NY 10012, USA}
\email{marques@cims.nyu.edu}
\def\dar[#1]{\ar@<2pt>[#1]\ar@<-2pt>[#1]}
  \entrymodifiers={!!<0pt,0.7ex>+}
\def\commutatifs{\ar@{}[lu]|{\circlearrowleft}}
\begin{document}

\setcounter{page}{1} 
\maketitle
\selectlanguage{english}
\setcounter{tocdepth}{4}
\footnotesize
\tableofcontents
\bibliographystyle{alpha-fr}
\pagestyle{plain}
\begin{abstract}
We study the ramification theory for actions involving group schemes, focusing on the tame ramification. We consider the notion of tame quotient stack introduced in [AOV] and the one of tame action introduced in [CEPT]. We establish a local Slice theorem for unramified actions and after proving some interesting lifting properties for linearly reductive group schemes, we establish a Slice theorem for actions by commutative group schemes inducing tame quotient stacks. Roughly speaking, we show that these actions are induced from an action of an extension of the inertia group on a finitely presented flat neighborhood. We finally consider the notion of tame action and determine how this notion is related to the one of tame quotient stack previously considered.
\end{abstract}
\footnotesize{
\subsection*{Keywords} Group schemes, Ramification, Unramified, Freeness, Tameness, Slice theorem, Linearly reductive group schemes, Lifting, Actions, Quotient stack, Trivial cohomology.
\subsection*{Acknowledgment} I would like to thanks my supervisors Boas Erez and Marco Garuti but also Angelo Vistoli and Yuri Tschinkel for their help and advise.}
\large
\section*{Introduction} 
It is well-known that locally, for the étale topology, 
actions of constant group schemes (and slightly more generally etale group schemes) are induced from actions of their inertia groups: 
one can reconstruct the original action from the action of an inertia group at a point 
(Theorem \ref{ringslice}) (this is a direct extension of the classic result of decomposition of finite extensions of valuation field after moving to the completion \cite[II, \S 3, Théorème 1, Corollaire 4 and  Proposition 4]{Serre}). This statement is an instance of a {\em slice theorem} 
(see Theorem \ref{constantslice}). 

In this paper, we establish a more general slice theorem, under a tameness hypothesis.  
One motivation for studying slices in this generality is the theory of 
tame covers, in the sense of Grothendieck and Murre: these admit étales slices 
(see Abhyankar's lemma in \cite{GM}, \cite{CE}). 
Another motivation is the fundamental theorem of Luna, which states that 
actions of linearly reductive algebraic groups on an affine varieties are tame (see \cite{Luna}) and admit slices.

In the simple case when the action has trivial inertia groups our 
slice theorem is simply the statement that freeness is local (Proposition \ref{locfree}).
More general tame quotient stacks were introduced by Abramovich, Olsson and Vistoli in \cite[\S 3]{AOV} 
in their study of certain ramification issues arising in moduli theory. 
Here, we characterize tameness for quotient stacks for actions of 
finite commutative group schemes via the existence of finitely presented flat slices 
with linearly reductive slice groups (Theorem \ref{commslice}). 
Roughly speaking, we show that these actions are induced from an action of 
an extension of the inertia group on a finitely presented flat neighborhood. 
Furthermore, this lifting of the inertia group can be constructed also in the non-commutative case, 
as a subgroup of the initial group and as a flat and linearly reductive group scheme (Theorem \ref{extension}). 
Moreover, \cite[Theorem 3.2]{AOV} shows that tameness is characterized by the property that 
all inertia groups at topological points are linearly reductive  
(it suffices to require this only at the geometric points).
 
This is what one could have expected as a definition of tameness. 
In fact, let us consider a $\Gamma$-extension of Dedekind rings $B/B^\Gamma$, 
where $\Gamma$ is an abstract finite group and $B^\Gamma$ 
is the invariant ring for the action of $\Gamma$ on $B$. 
For any prime ideal $\mathfrak{p}$ of $B$ we write $k( \mathfrak{p})$; 
then, it is well known, that $B/B^\Gamma$ is tame if and only if for all prime $\mathfrak{p}$ of $B$, the inertia group $\Gamma_0 (\mathfrak{p})$ has order prime to the characteristic of $k ( \mathfrak{p})$. 
This last condition is equivalent to requiring that the group algebra 
$ k ( \mathfrak{p}) [\Gamma_0 ( \mathfrak{p})]$ be semisimple, i.e., 
that the constant group scheme $\underline{\Gamma_0 (\mathfrak{p})}$ attached to $\Gamma_0( \mathfrak{p})$ is linearly reductive. 
But as $k( \mathfrak{p})$-group, $\underline{\Gamma_0 (\mathfrak{p})}$ is exactly the inertia group of
 the action of the constant group scheme attached to 
$\Gamma$ on $X=Spec(B)$ (cf \cite[III, \S 2, n° 2, Example 2.4]{DemGab}).  \\  

Independently, in \cite[\S 2]{Boas}, Chinburg, Pappas, Erez and Taylor introduced the notion of {\em tame actions}. 
Thus it is natural to anticipate a relationship between these two notions of tameness. 
We prove that under quite general hypotheses, 
{\em tame actions} define {\em tame quotient stacks} (Theorem \ref{inertiefini}). 
In fact, under additional hypotheses such as finiteness, the two notions are equivalent (Theorem \ref{proppal}). 
Thus, the previous results apply to this notion of tameness and we can
answer more precisely \cite[\S 4, Question $2$ and $3$]{Boas}.\\

\section*{Hypotheses, basic concepts and notation}

We write fppf for faithfully flat and finitely presented. Throughout, we fix the following notation.

\

Let $R$ be a commutative, noetherian, unitary base ring:  all modules and algebras are over $R$ and 
all algebras are commutative. Let $S:=Spec(R)$ be the corresponding affine base scheme:  
all schemes are $S$-schemes. For $V$ a scheme over $S$ and $S'$ an $S$-scheme, let 
$V_{S'}$ (resp. ${}_{S'} V$) be the base change $V\times_S S'$ (resp. $S'\times_S V$). 
In particular, for $V$ a scheme over $S$ and $R'$ an $R$-algebra, we write $V_{R'}$ (resp. ${}_{R'} V$) instead of $V_{Spec(R')}$ (resp. ${}_{Spec(R')} V$) for the base change $V\times_S Spec(R')$ (resp. $Spec(R') \times_S V$). 
Given $S$-morphisms of schemes $V \rightarrow W$ and $Z\rightarrow W$, we have the fiber product

\centerline{
\xymatrix{
V\times_W Z \ar[d]_-{pr_1} \ar[r]^-{\,\,\,\,pr_2} & Z \ar[d]\\
 V \ar[r]    & W 
}
}  

\

\noindent
together with the natural projections.


\

Let $A=(A, \delta, \epsilon, s)$ be a flat, finitely presented commutative Hopf algebra over $R$ where $\delta$ denotes the comultiplicaton, $\epsilon$ the unit map and $s$ the antipode, 
$G$ the affine flat group scheme associated to $A$ over $S$ and $X:=Spec(B)$ an affine scheme over $S$. \\ 

\begin{remarque*} 
For simplicity, we only consider actions involving affine schemes even though most of the following results are true for actions involving general schemes by glueing. Only the definition of tameness by \cite{Boas} cannot be always generalized for actions involving the non-affine schemes (see \cite[\S 7]{Boas}). 
\end{remarque*}

An action of $G$ on $X$ over $S$ is denoted by $(X,G)$; we write 
$\mu_X :X\times_S G \rightarrow X$ for
its structure map, (we write $x \cdot g := \mu_X(x, g)$, for any $x \in X$ and $g\in G$) and $\rho_B : B \rightarrow B \otimes_RA$ for
the structure map giving the $A$-comodule structure on $B$ (see \cite[\S 1.b]{Boas}). \\ 

We will use the  \textbf{sigma notation}, well-established in the Hopf algebra literature. 
Specifically, for $a \in A$, we write $\delta (a) = \sum_{(a)} a_1 \otimes a_2$. This presentation itself is purely symbolic; the terms $a_1$ and $a_2$ do not stand for particular elements of A. The comultiplication $\delta$ takes values in $A\otimes_R A$, and we know that:
$$ 
\delta (a) = (a_{1,1} \otimes a_{1,2}) + (a_{2,1} \otimes a_{2,2}) + (a_{3,1} \otimes a_{3,2}) + \ldots + (a_{n,1} \otimes a_{n,2}),
$$
for some elements $a_{i,j}$ of $A$ and some integer $n$. 
The Sigma notation is a way to separate the $a_{i,1}$ from the $a_{j,2}$. In other words, 
the notation $a_1$ stands for the generic $a_{i,1}$ and the notation $a_2$ stands for the generic $a_{j,2}$. Similarly, for any right comodule (respectively left comodule) $M$, any $m \in M$, we write 
$\rho_M (m) =\sum m_{(0)} \otimes m_{(1)}$ (respectively ${}_M\rho(m) = \sum m_{(-1)}\otimes m_{(0)}$). \\

For $\mathfrak{p} \in X$, we denote $I_G ( \mathfrak{p})$ the \textbf{inertia group at $\mathfrak{p}$} to be the fiber product: 
$$\xymatrix{I_G (\mathfrak{p}) \ar[r]^{pr_2} \ar[d]_{pr_1} & Spec( k(\mathfrak{p}) )\ar[d]^{\Delta \circ \xi } \\ X \times_S G \ar[r]_{(\mu_X, pr_1)} & X \times_S X }$$
where $k(\mathfrak{p})$ is the residue field at $\mathfrak{p}$, $\xi : Spec(k(\mathfrak{p})) \rightarrow X$ is the morphism induced by the canonical morphism $B \rightarrow k(\mathfrak{p})$ and $(\mu_X , pr_1): X \times_S G \rightarrow X\times_S X$ is the \textbf{Galois map} sending $(x, g)$ to $( x\cdot g, x)$.\\ 

Let 
$$
C:=B^A= \{ b \in B | \rho_B (b) = b \otimes 1 \}
$$ 
be the ring of invariants for this action, $Y:= Spec (C)$ and $\pi : X \rightarrow Y$ 
the morphism induced by the inclusion $C\subseteq B$. \\
\begin{remarque*} We can also then consider $X$ as a scheme over $Y$. Then the data of the $S$-scheme morphism $\mu_X : X\times_S G\rightarrow X$ defining an action $(X, G)$ over $S$ is equivalent to the data of the $Y$-scheme morphism $\mu_X : X\times_Y G_Y \rightarrow X$ defining an action $(X, G_Y)$ over $Y$. For $Y'$ a $Y$-scheme, we denote $X_Y^{Y'}$ the base change to $Y'$ for $X$ considered as a $Y$-scheme.
\end{remarque*}

In the following, we will say that a $U$-scheme $Z$ together with a $G$-action over $S$ where $U$ is a $S$-scheme is a \textbf{$G$-torsor over $U$} if $Z \rightarrow U$ is fppf and the Galois map $(\mu_Z , pr_1): Z\times_S G \simeq Z \times_U Z$ is an isomorphism. (Here, we choose to work with the fppf topology which is the more reasonable topology to work under our assumptions). \\ 

We write $[X/G]$  for the quotient stack associated to this action. Recall, that it is defined on the $T$-points, where $T$ is a $S$-scheme as: 
$$[X/G](T) = \{ P \text{ is } G\text{-torsor over } T \text{ together with a $G$-equivariant morphism } P\rightarrow X\}$$
We get a canonical $S$-morphism $p: X \rightarrow [X/G]$ by sending a $T$-point $T \rightarrow X$ to trivial $G$-torsor $T \times_S G$ over $T$ together with the $G$-equivariant morphism $T\times_S G \rightarrow X$. Under the present hypothesis, $p$ is fppf and by the Artin criterium (see \cite[Théorème 10.1]{Laumon}), $[X/G]$ is an Artin stack. We denote $B_SG$ the \textbf{classifying stack} that is the quotient stack associated to the trivial action of $G$ on $S$. \\ 

Moreover, $X/G$ stands for the categorical quotient in the category of algebraic spaces, that is $X/G$ is an algebraic space, there is a morphism $\phi : X\to X/G$ such that $\phi\circ\mu_X=\phi\circ pr_1$ and for any algebraic space $Z$ and any morphism $\psi:X\to Z$ such that $\psi\circ\mu_X=\psi\circ pr_1$, $\psi$ factorizes via $\phi$, e.g. there is a morphism $X/G \rightarrow Z$ making the following diagram commute: 
$$\xymatrix{ X \ar[d]_{\phi} \ar[r]^{\psi} & Z \\ X/G \ar[ur] & }$$
If $G$ is finite, flat over $S$ then $X/G= Y$, by \cite[Theorem 3.1]{conrad}.
\begin{remarque*} The quotient stack $[X/G]$ can always be defined, as soon as $G$ is a flat group scheme, instead  $X/G$ does not exist necessarily. Even if $X/G$ does exist, the quotient stack $[X/G]$ gives more information about the action than the categorical quotient $X/G$. More precisely, the only difference between the two holds in the existence of automorphisms of points for the stack $[X/G]$, but these automorphisms correspond to the inertia groups for the action $(X,G)$. A slice theorem reducing, fppf locally, the action $(X,G)$ to an action by a lifting of an inertia group would imply that the data of the quotient stack associated to $X$ and $G$ is enough to rebuild locally the action $(X,G)$. This is the case for action of finite and étale group scheme (see Theorem \ref{ringslice}).
\end{remarque*}
We denote  by $\mathcal{M}^A$ (resp. ${}^A  \mathcal{M}$) the category of 
the right (resp. left) $A$-comodules  and by ($\mathcal{M}_R$ (resp. ${}_R \mathcal{M}$) 
the category of right (resp. left) $R$-modules). Let $(M,\rho_M )$ and $(N, \rho_N)$ be two right $A$-comodules. 
Denote by  $Com_A(M,N)$ the $R$-module of the $A$-comodule morphisms from $M$ to $N$. We recall that a $R$-linear map $g :M \rightarrow N$ is called \textbf{$A$-comodule morphism} if $\rho_N \circ g = (g \otimes Id_A) \circ \rho_M$.\\

We denote by ${}_B\mathcal{M}^A$ the category of $\textbf{(B,A)}$\textbf{-modules}. The objects are $R$-modules $N$ which are also right $A$-comodules and left $B$-modules such that the structural map of comodules is $B$-linear, i.e., 
$$
\rho_N (bn) =\rho_B(b) \rho_N(n)=\sum b_{(0)} n_{(0)} \otimes  b_{(1)}n_{(1)}, \quad \text{ for all } \quad 
n \in N, b\in B.
$$ 
The morphisms of ${}_B\mathcal{M}^A$ are morphisms that are 
simultaneously $B$-linear and $A$-comodule morphisms. 
For any $M, N\in {}_B\mathcal{M}^A$, we write ${}_BBim^A(M,N)$ 
for the set of such morphisms from $M$ to $N$.  
For  a $(B,A)$-module $M$, its \textbf{submodule of invariants} $(M)^A$ 
is defined by the exact sequence
$$
0 \rightarrow (M)^A \rightarrow  \raisebox{.7ex}{\xymatrix{ M \dar[r]^-*[@]{\hbox to 0pt{\hss\txt{ $\rho_M$  \\ }\hss}}_-{M\otimes 1} & M \otimes A  }}
$$ 
This defines the \textbf{functor of invariants} 
$$
(-)^A : {}_B\mathcal{M}^A \rightarrow {}_C\mathcal{M}
$$ 
which is left exact, by definition.
We denote $Qcoh^G (X)$ the $G$-equivariant quasi-coherent sheaves over $X$ (see \cite[2.1]{AOV}). We have that $Qcoh^G (X)\simeq {}_B\mathcal{M}^A$. 
For $\mathcal{X}$ an Artin stack, we denote by $Qcoh( \mathcal{X})$ the quasi-coherent sheaves over $\mathcal{X}$. \\ 
We have $Qcoh( [X/G]) \simeq Qcoh^G(X)$ (see \cite[2.1]{AOV} or  \cite[Chapter IV, \S 9]{these}).
\section{Definition of slices}
Here we recall the notion of slices introduced in \cite[Definition 3.1]{Boas}: 

\begin{definition} 
\label{slice}
We say that the action $(X,G)$ over $S$ admits \textbf{étale (respectively finitely presented flat) slices }if: 
\begin{enumerate} 
\item There is a categorical quotient $Q:=X/G$, in the category of algebraic spaces. 
\item For any $\mathfrak{q} \in Q$, there exist: 
\begin{enumerate} 
\item a $S$-scheme $Q'$ and an étale (resp. finitely presented flat) $S$-morphism $Q' \rightarrow Q$ such that there is $\mathfrak{q}' \in Q'$ which maps to $\mathfrak{q}$ via this morphism. 
\item a closed subgroup $G_\mathfrak{q}$ of $G_{Q'}$ over $Q'$ 
which stabilizes some point $\mathfrak{p}' $ of $X \times_Q Q'$ above $\mathfrak{q}' $, i.e.,  ${G_\mathfrak{q} }_{k(\mathfrak{p} ')}\simeq I_{G_{Q'}}(\mathfrak{p}' )$,
\item a $Q'$-scheme $Z$ with a $G_\mathfrak{q}$-action such that 
$Q' =Z/{G_\mathfrak{q}}$ and the action $(X\times_Q Q', G_{Q'})$ is induced by $( Z , G_\mathfrak{q})$.
\end{enumerate}
The subgroup $G_\mathfrak{q}$ is called a \textbf{slice group}.
\end{enumerate} 
\end{definition}

\begin{remarque}
\begin{enumerate}
\item Roughly speaking, for the étale (respectively fppf) topology, 
an action which admits slices can be described by the action of a lifting of an inertia group of a point.
\item The action $(X_{Q'}, G_{Q'})$ should be thought of as a $G$-stable neighborhood of an orbit. 
Such a neighborhood induced from an action $(Z, H)$, 
where $H$ is a lifting of an inertia group of a point, is called a \textit{\og tubular neighborhood \fg}.
\end{enumerate}
\end{remarque}

\section{Slice theorem for actions by finite étale (smooth) group scheme}
Since a finite and smooth group scheme is étale, and an étale group scheme is locally constant for the étale topology, it is enough to consider an action by a constant group scheme. \\ 
Let $\Gamma$ be an abstract finite group. Denote by 
$\underline{\Gamma}$ the constant group scheme associated to $\Gamma$ over $S$. The associated Hopf algebra is then of the form $Map ( \Gamma , R)$. 
We recall first some important facts:
\begin{enumerate} 
\item The data of an action of $G$ on $X$ is equivalent to the data of an 
action of $\Gamma$ on $B$, $C= B^\Gamma = \{ \gamma \in \Gamma | \gamma b = b \}$ (see \cite[\S 1.d]{Boas}) and $X/ \underline{\Gamma}= Y$ (see \cite[Theorem 3.1]{conrad}). 
This particular case also permits to make the transition from algebra and classic number theory to the algebraic geometry context and to understand later definitions.
\item For $\mathfrak{q} \in Y$, $\Gamma$ operates transitively 
on the prime ideals above $\mathfrak{q}$ (see \cite[Chap. V, \S 2, Theorem 2]{Bourbaki}).
\item The inertia group scheme at a point $\mathfrak{p} \in X$ is 
the constant group scheme associated to the algebraic inertia group 
$$\Gamma_0 ( \mathfrak{p})=\{ \gamma \in \Gamma | \gamma \mathfrak{p} = \mathfrak{p} \text{ and } \gamma . x= x, \  \forall x \in k(\mathfrak{p}) \}$$ for the action $\Gamma$ on $B$ induced by $(X,G)$ at the ideal 
$\mathfrak{p}$ (see  \cite[III, \S 2, n° 2, Example 2.4]{DemGab}).
\end{enumerate}
In order to alleviate writing, we use that the morphism $B \rightarrow C$ is integral (see \cite[Chap. V, \S1, n°9, Proposition 22]{Bourbaki}).Let $\mathfrak{q}$ be a prime ideal on $C$, if we write $C_{\mathfrak{q}}^{sh}$ for the strict henselization of $C$ at $\mathfrak{q}$, then the base change $B \otimes_C C_{\mathfrak{q}}^{sh} \rightarrow C_{\mathfrak{q}}^{sh}$ is also integral 
 and since we want a result locally for the étale topology, without loss of generality, we suppose in the following that the base is local strictly henselian with maximal ideal $\mathfrak{q}$ and that $B$ is semi-local over $C$ that is, the product of its local component $(B_\mathfrak{m})$ where $\mathfrak{m}$ are runs though the set of the prime (maximal) ideals of $B$ over the maximal ideal $\mathfrak{q}$ of $C$. Let $\mathfrak{p}$ be a prime ideal over $\mathfrak{q}$. We define $Map^{\Gamma_0(\mathfrak{p})} (\Gamma , B_{\mathfrak{p}})$ to be the set of maps 
$$
u :\Gamma\to B_{\mathfrak{p}}
$$ 
such that $u(\gamma i) =i^{-1}u(\gamma)$, for any $ \gamma \in \Gamma$ and $i \in \Gamma_0 ( \mathfrak{P})$;
it carries a $\Gamma$-action via 
$$\lambda . u( \gamma)= u(\lambda^{-1} \gamma ),\quad \text{ for any }\quad 
\lambda, \gamma \in \Gamma.
$$ 
With these notations, we obtain easily the following $\Gamma$-equivariant isomorphism: 
$$(Map(\Gamma, C) \otimes_C B_\mathfrak{p})^{\Gamma_0 ( \mathfrak{p})} \simeq Map^{\Gamma_0(\mathfrak{p})}(\Gamma, B_\mathfrak{p})$$
With all the previous notations, one obtains that the action of $\Gamma$ on $B$ can be rebuilt locally at $\mathfrak{q}$ thanks to the action of an inertia group at $\mathfrak{p}$ on $B_\mathfrak{p}$:
\begin{lemme}\label{ringslice}\cite[Chapitre X]{raynaud}
We have a canonical isomorphism $\phi$, compatible with the actions of $\Gamma$ defined by $$\begin{array}{lrll} \phi : &  B& \rightarrow & Map^{\Gamma_0 ( \mathfrak{p})} (\Gamma ,  B_{\mathfrak{p}})\\ & b & \mapsto & u : \gamma \mapsto ( \gamma^{-1} b )_\mathfrak{p} \end{array}$$
Moreover, $C \simeq B_\mathfrak{p}^{\Gamma_0( \mathfrak{p})}$.
\end{lemme}
\begin{proof} 
Since $B_\mathfrak{p} \simeq i B_\mathfrak{p}$, for any $i \in \Gamma_0 ( \mathfrak{p})$, $\phi$ is well-defined.
Let $\mathfrak{p}_1$, ....,$\mathfrak{p}_s$ be the prime ideals over $\mathfrak{q}$, we put $\mathfrak{p}_1:= \mathfrak{p}$. Since the action of $\Gamma$ is transitive on the $\mathfrak{p}_i$'s, for any $i \neq 1$, we have a $\gamma_i \in \Gamma$, such that $\mathfrak{p}_i = \gamma_i \mathfrak{p}$. Moreover, $B_{\mathfrak{p}_j}\simeq\gamma_j B_\mathfrak{p}$. 
Since we suppose $C$ strictly henselian and $B$ is integral over $C$, then we know that $B = \oplus_{i=1}^s B_{\mathfrak{p}_i}\simeq\oplus_{i=1}^s \gamma_i B_{\mathfrak{p}}$. As a consequence, any $b\in B$ can be written uniquely as $({\gamma_1}b_1 , {\gamma_2} b_2, ...,{ \gamma_s} b_s)$ in $\oplus_{i=1}^s \gamma_i B_{\mathfrak{p}}$ with $b_i \in B_\mathfrak{p}$. Each $\gamma_i$ acts on this representation, by left translation and permuting the term on the direct sum.
In particular, we obtain $(\gamma_j^{-1} b)_{\mathfrak{P}} = \gamma_j^{-1} \gamma_j b_j = b_j$,  and this proves that $\phi$ is an isomorphism.
Finally, the composite map $$\begin{array}{ccccc}B &\simeq^{\phi} &Map^{\Gamma_0 ( \mathfrak{p})}(\Gamma , B_\mathfrak{p})& \rightarrow &B_\mathfrak{p} \\ &&u & \mapsto & u(1_\Gamma )\end{array}$$ induces an isomorphism between $C$ and $B_\mathfrak{p}^{\Gamma_0 ( \mathfrak{p})}$.
\end{proof}
Rewriting the previous theorem in algebro-geometric terms gives us exactly étales slices for any action by a finite étale group scheme.
\begin{theoreme}\label{constantslice}
An action of a finite étale group scheme $G$ on $X$ over $S$ admits étale slices.  
\end{theoreme}
\begin{proof} First, $Y=X/G$ is a categorical quotient in the category of algebraic space (see \cite[Theorem 3.1]{conrad}).  The rest of the proof is a direct consequence of the previous theorem since as we already mentioned we can assume without loss of generality that $G$ is constant of the form $\underline{\Gamma}$ using the previous notation: take $\mathfrak{q} \in Y$ and $\mathfrak{p} \in X$  over $\mathfrak{q}$. 
Write $C_{\mathfrak{q}}^{hs}$ for the strict Henselization of $C$ at $\mathfrak{q}$. In Definition \ref{slice}, since $X\rightarrow Y$ is finite, we can take $Y'=Spec(C')$ as an étale subextension of $Spec (C_{\mathfrak{q}}^{hs})$ over $Y$ containing $\mathfrak{q}$ in its image with $C'$ a local ring, the slice group $G_\mathfrak{q}$ to be the constant group scheme associated to $\Gamma_0 ( \mathfrak{p})$ over $Y'$ and $Z :=Spec ((B\otimes_C C')_\mathfrak{m})$, where $\mathfrak{m}$ is a prime ideal above the maximal ideal of the local ring  $C'$ corresponding to $\mathfrak{q}$.
\end{proof}
\section{Unramified case}
\subsection{Definitions} 
\begin{definition}
We say that the action $(X,G)$ is \textbf{unramified at $\mathfrak{p}\in X$} if the inertia group $I_G(\mathfrak{p})$ is trivial. We say that the action $(X,G)$ is \textbf{unramified} if $(X, G)$ is unramified at any $\mathfrak{p}$. We say the action is free if the Galois morphism $(\mu_X, pr_1): X\times_SG \rightarrow X\times_S X$ is a closed immersion. 
\end{definition}
\begin{lemme} \cite[III, \S 2, n° 2, 2.2]{DemGab}
The following assertions are equivalent: 
\begin{enumerate}
\item $(X, G)$ is free.
\item  $(X, G)$ is unramified.
\end{enumerate}
\end{lemme}
\subsection{Global slice theorem} We already have the following theorem. 
\begin{theoreme}
 \cite[III, \S 2, n°2 \& n°3]{DemGab} \label{locfree} 
Suppose that $G$ is finite, flat over $S$. The following assertions are equivalent: 
\begin{enumerate} 
\item $(X,G)$ is unramified. 
\item $X$ is a $G$-torsor over $Y$. 
  \end{enumerate}
\end{theoreme}
\subsection{Local slice theorem}
\begin{lemme}  Suppose that $S=Y$ is local, with maximal ideal $\mathfrak{q}$. The following assertions are equivalent:
\begin{enumerate}
\item $(X,G)$ is a free action. 
\item $I_G( \mathfrak{p})$ is trivial for some $\mathfrak{p}$ over $\mathfrak{q}$.
\end{enumerate}
\end{lemme}
\begin{proof} The first implication follows from the definition. 
Let us prove that $(2) \Rightarrow (1)$. Let $k (\mathfrak{q})= R/\mathfrak{q}$ be the residue field of $R$ at $\mathfrak{q}$ and $f : X \rightarrow S$ be the quotient morphism: since $G$ is finite, $f$ is finite. Set $S_0:= Spec(k (\mathfrak{q}))$ and $X_0 := f^{-1} (\mathfrak{q}) = \{ \mathfrak{p}_1 , .., \mathfrak{p}_r\}= X \times_S S_0$ where for any $ i \in 1,..., r, \ \mathfrak{p}_i$ are the primes above $\mathfrak{q}$. Since the inertia is trivial at some prime over $\mathfrak{q}$, it is trivial at any prime over $\mathfrak{q}$ since all the prime ideals over $\mathfrak{q}$ are conjugate up to a finitely presented flat finite base change. Then, up to this finitely presented flat finite base change, the action $(X_{S_0} , G_{S_0})$ is free, in other words $X_0 \times_{S_0} G_{S_0} \rightarrow X_0 \times_{S_0} X_0$ is a closed immersion (see \cite[III, \S 2, n°2, Proposition 2.2]{DemGab}). That is, we have the surjection 
$$
(B\otimes_R B) \otimes_R R/\mathfrak{q} \rightarrow (B \otimes_R A) \otimes_R R/ \mathfrak{q}.
$$ 
Since $B\otimes_R B$ and $B\otimes_RA $ are finite $R$-algebras, by Nakayama's lemma, 
$B\otimes_R B \rightarrow B \otimes_RA$ is a surjection, hence the action is free.
\end{proof}
From this lemma, we can deduce easily the following theorem which is a local slice theorem for free action:
\begin{theoreme} \label{locfree0} Suppose that $G$ is finite and flat over $S$. Let $\mathfrak{p} \in X$ and $\mathfrak{q}\in Y$ its image via the morphism $\pi : X \rightarrow Y$. The following assertions are equivalent: 
\begin{enumerate} 
\item The inertia group scheme is trivial at $\mathfrak{p}$, up to a finite base change. 
\item There is an finite, flat morphism $ Y' \rightarrow Y$ containing $\mathfrak{q}$ in its image such that the action $(X_{Y'}^Y , G_{Y'})$ over $Y'$ is free. In other words, $(X_{Y'}^Y \times_Y X_{Y'}^Y, G_{X_{Y'}})$ is induced by the action $(X_{Y'}^Y \times_{Y'} X_{Y'}^Y,e_{X_{Y'}^Y})$, where $e_{X_{Y'}^Y}$ denotes the trivial group scheme over $X_{Y'}^Y$.  \end{enumerate}
\end{theoreme}

\begin{proof}
$(1) \Rightarrow (2)$ From the previous lemma, there is an finite, flat morphism $Y' \rightarrow Y$ containing $\mathfrak{q}$ in its image such that $(X_{Y'}^Y, G_{Y'})$ is free, and since $G_{Y'}$ is finite over $Y'$, $X_{Y'}$ is a $G_{Y'}$-torsor over $Y'$ (see Theorem \ref{locfree}) hence  $X_{Y'}^Y$ is an fppf morphism over $Y'$, there is a $G_{Y'}$-equivariant isomorphism 
$$(\mu_{X_{Y'}^Y} , pr_1 ): (X_{Y'}^Y \times_{Y'} G_{Y'})/ e_{X_{Y'}^Y} \simeq X_{Y'}^Y \times_{Y'} X_{Y'}^Y$$
where on the left hand side $G_{Y^\prime}$ acts only on the factor $G_{Y^\prime}$ by translation and on the right hand side $G_{Y^\prime}$ acts on the first factor only. \\ 
$(2) \Rightarrow (1)$ Trivial.
\end{proof}
\section{Linearly reductive group scheme}
\subsection{Definition}
From the notion of tameness, 
we can define a class of group schemes 
with nice properties, which will be very important for the following.
\begin{definition}
We say that $G$ is \textbf{linearly reductive group scheme}, if $(-)^A : {}_R \mathcal{M}^A \rightarrow {}_R \mathcal{M}$ is exact. 
\end{definition}
\subsection{Cohomological properties}
The cohomology of linearly reductive group schemes has some very interesting cohomological vanishing properties that are useful for deformation involving such group schemes. We can even characterize linearly reductivity by the vanishing of the cohomology:
\begin{lemme} (see \cite[Proposition 1]{Kemper})
Suppose that $G$ is flat, finite over $S$, that $X$ finitely presented over $S$ and $S = Spec(k)$ where $k$ is a field.
The following assertion are equivalent: 
\begin{enumerate}
\item $G$ is linearly reductive over $S$. 
\item $H^1 (G, M ) = 0$ for any $M \in \mathcal{M}^A$. 
\end{enumerate}
\end{lemme}
The following lemma is part of the proof of \cite[Lemma 2.13]{AOV}:
\begin{lemme}\label{smoothlr}
If $G$ is linearly reductive over $S$, then 
\begin{enumerate}
\item $$ 
Ext^i ( L_{B_k G/k} , \mathcal{F})= 0 , \ for \ i \neq -1, 0.
$$
for any coherent sheaf $\mathcal{F}$ on $B_kG$.
\item If we suppose that $G$ is also smooth, 
$$Ext^i ( L_{B_k G/k} , \mathcal{F}) =0, \  for\  i\neq -1 $$ 
for any coherent sheaf $\mathcal{F}$ on $B_kG$.
\end{enumerate}
\end{lemme}
\begin{proof}
By \cite[Lemma 2.15]{AOV}, the cotangent complex $L_{B_k G/k}$ of the structural morphism $B_k G\rightarrow k$ (where $k$ is a field) belongs to $D_{coh}^{[-1,0]} ( \mathcal{O}_{B_kG})$. Since $k$ is a field, any coherent sheaf on $B_kG$ is locally free, and therefore for any coherent sheaf $\mathcal{F}$ on $B_kG$, we have: 
$$
RHom ( L_{B_k G/k} , \mathcal{F}) \in D_{coh}^{[-1,0]} ( \mathcal{O}_{B_kG}).
$$ 
\begin{enumerate}
\item Since the global section functor is exact on the category $Coh( \mathcal{O}_{B_kG})$ (since $G$ is linearly reductive), we obtain that  then 
$$ 
Ext^i ( L_{B_k G/k} , \mathcal{F})= 0 , \ for \ i \neq -1, 0.
$$
for any $\mathcal{F}$ is a coherent sheaf on $B_kG$.
\item When $G$ is smooth, then $L_{B_k G/k}\in D_{coh}^{[0]} ( \mathcal{O}_{B_kG})$ and 
$$Ext^i ( L_{B_k G/k} , \mathcal{F}) =0, \  for\  i\neq -1 $$ 
for any $\mathcal{F}$ is a coherent sheaf on $B_kG$.
\end{enumerate}
\end{proof}
\subsection{Liftings of linearly reductive group schemes} 
By \cite[Proposition 2.18]{AOV}, we know that a linearly reductive group scheme over some point can be lifted as linearly reductive group to an étale neighborhood. Thanks to this result, we were able to prove that we can lift this group fppf locally as a subgroup in the following sense. 
\begin{theoreme} \label{extension}
Let $\mathfrak{p}$ be a point of $S$, $G$ be a finite, flat group scheme over $S$ and $H_0$ be a finite, linearly reductive closed subgroup scheme of $G_{k( \mathfrak{p})}$ over $Spec(k( \mathfrak{p}))$. Then, there exists a flat, finitely presented morphism $U \rightarrow S$, 
a point $\mathfrak{q} \in U$ mapping to $\mathfrak{p}$ and a flat linearly reductive closed subgroup scheme $H$ of $G_U$ over $U$ whose pullback $H_{k( \mathfrak{q} )}$ is isomorphic to the pullback of ${H_0}_{k(\mathfrak{q})}$. 
\end{theoreme}
\begin{proof} 
Let $\mathfrak{p} \in S$. By \cite[Proposition 2.18]{AOV}, there exists an étale morphism 
$U \rightarrow S$, a point $\mathfrak{q} \in U$ mapping to $\mathfrak{p}$ and a linearly reductive group scheme $H$ over $U$ such that $H_{k( \mathfrak{q})} \simeq {H_0}_{k(\mathfrak{q})}$. Set $U_n := Spec ( R / \mathfrak{q}^{n+1})$. One has that $H_0$ is a subgroup scheme of $G_{k( \mathfrak{p})}$, and this defines a representable morphism of algebraic stacks $x_0 : B_{k( \mathfrak{q})} H_0\rightarrow B_{k(\mathfrak{q})} G_{k( \mathfrak{q})}$ over $k(\mathfrak{q})$. We want to prove the existence a representable morphism of algebraic stacks $x: B_U H \rightarrow B_U G_U$ filling in the following $2$-commutative diagram
$$
 \xymatrix {
B_{k( \mathfrak{q})} H_0 \ar[rr]  \ar[dddr] \ar[dr]^x && B_U H\ar@{.>}[dr]^{x'} \ar[dddr] \\
    &B_{k(\mathfrak{q})} G_{k( \mathfrak{q})} \ar[rr] \ar[dd]^g &&  B_U G_U\ar[dd] \\
    & &  \\
    & U_0 \ar[rr] && U \\
  }$$
We can prove, thanks to Grothendieck 's Existence theorem for algebraic stacks (see \cite[Theorem 11.1]{Olsson}) and Artin's approximation theorem (see \cite[Théorème 1.12]{Artin2}), that the existence of such $x$ only depends on the existence of a formal deformation, that is morphisms $x_n: B_{U_n} H_{U_n} \rightarrow B_{U_n} G_{U_n}$ filling in the following $2$-commutative diagram: 
$$ \xymatrix {
   B_{k( \mathfrak{q})} H_0   \ar@{.>}[rr]^{i_n} \ar[dddr] \ar[dr]^{x_1} &&  B_{U_{n-1}} H_{U_{n-1}} \ar[rr] \ar[dddr] \ar[dr]^{x_{n-1}} && B_{U_n} H_{U_n}\ar@{.>}[dr]^{x_n} \ar[dddr]\\
    &B_{k(\mathfrak{q})} G_{k( \mathfrak{q})}  \ar@{.>}[rr]^{j_n} \ar[dd] &&B_{U_{n-1}} G_{U_{n-1}} \ar[rr] \ar[dd]&& B_{U_n} G_{U_n} \ar[dd]  \\
    & &  \\
    & Spec(k( \mathfrak{q})) \ar@{.>}[rr] &&U_{n-1} \ar[rr] && U_n \\
  }$$
 By \cite[Theorem 1.5]{Olsson2}, the obstruction to extend the morphism $x_{n-1}$ to $x_n$ lies in 
$$
Ext^1 ( Lx^*L_{B_{k( \mathfrak{q}) }G_{k( \mathfrak{q}) } /k( \mathfrak{q})}, \mathfrak{m}^{n}/ \mathfrak{m}^{n+1} \otimes_{k(\mathfrak{q})} \mathcal{O}_{B_{k(\mathfrak{q})} G_{k( \mathfrak{q})}}),
$$ 
which is trivial for any $n \in \mathbb{N}$ by Lemma \ref{smoothlr}. It follows that there exists an arrow $x_n$ filling the previous diagram.
This leads to the existence of a representable morphism of stacks $F :B_U H_U \rightarrow B_U G_U$. 
Let $Q\rightarrow U$ be the $G_U$-torsor 
which is the image via $F$ of the trivial $H_U$-torsor $H_U\rightarrow U$. 
Furthermore, the functor $F$ induces a homomorphism from $\underline{Aut}_S ( H_Q \rightarrow Q )= H_Q$ to the automorphism group scheme of the image of the $H_Q$-torsor $H_Q \rightarrow Q$ in $B_Q G_Q$. Since this image is the pullback of $Q$ to $Q$ over $B_Q G_Q$, which is canonically a trivial torsor, its automorphism group is $G_Q$. This defines a group morphism $H_Q \rightarrow G_Q$, 
with $Q\rightarrow U$ an fppf morphism as $G_U$-torsor. Since $F$ is representable, this morphism is injective. 
Finally, since $H$ is proper and $G$ is separated, $H_V$ is closed in $G_V$.  
\end{proof}
\section{Tame quotient stack}
\subsection{Coarse moduli spaces}
Abramovich, Olsson and Vistoli introduced in \cite{AOV} the notion of tame stack. 
Before recalling their definition, we need some additional terminology. 
\begin{definition} 
\textbf{A coarse moduli space for the quotient stack $[X/G]$} is a couple $(M, \phi)$ where $M$ is a algebraic space and $\phi : [X/G]\rightarrow M$ a morphism of stacks such that any morphism from $[X/G]$ to an algebraic space factor through $\phi$ and such that for any algebraically closed field $\Omega$, $|[X/G]( \Omega )| \simeq M (\Omega )$ where $|[X/G](\Omega )|$ is the set of 
isomorphism classes of the geometric points of $[X/G]$ taking value on $\Omega$. 
\end{definition}
\begin{remarque}\label{quotient}
\begin{enumerate}
\item The notion of coarse moduli spaces can be generalized for general Artin stack. Coarse moduli space for quotient stack are in particular a categorical quotient in the category  of algebraic spaces for the action $(X,G)$. More precisely, 
\begin{enumerate}
\item  Let $M$ be an algebraic space. The datum of a morphism $\phi :[X/G] \rightarrow M$ is equivalent to the datum of a morphism of $S$-algebraic spaces $f: X \rightarrow M$ such that $f\circ pr_1 = f \circ \mu_X $. In particular, the canonical map $\pi : X \rightarrow Y$ induces a map $\rho : [X/G] \rightarrow Y$. 
\item If $f : X \rightarrow M$ is a geometric quotient which is categorical in the category of algebraic spaces, then $[X/G] \rightarrow M$ is a coarse moduli space. 
\end{enumerate}
For more details, see \cite[10.4]{these}.
\item If $G$ is finite, flat over $S$ then $\rho : [X/G]\rightarrow Y$ is a coarse moduli space for $[X/G]$ (see \cite[\S 3]{conrad}).
\end{enumerate}
\end{remarque}
\begin{theoreme}\label{coarseexistence}
Suppose that the group scheme $G$ and the scheme $X$ are finitely presented over $S$ and that all the inertia group schemes are finite. The quotient stack $[X/G]$ admits a coarse moduli space $\phi : [X/G] \rightarrow M$ such that $\phi$ is proper. Moreover, $\phi$ induces a functor $\phi_* : Qcoh( [X/G] ) \rightarrow Qcoh (M)$.
\end{theoreme}
\begin{proof}
Since $G$ is finitely presented over $S$, so is $X\rightarrow [X/G]$ (X is a $G$-torsor over $[X/G]$). Since $X\rightarrow [X/G]$ is surjective, flat of finite presentation and $X \rightarrow S$ is also of finite presentation, $[X/G]\rightarrow S$ is of finite presentation. By \cite[Theorem 1.1]{conrad}, the hypotheses insure that the quotient stack admits a coarse moduli space that we denote $\phi : [X/G] \rightarrow M$ such that $\phi$ is proper. In particular, 
$\phi$ is quasi-compact and quasi-separated, thus by \cite[Lem. 6.5(i)]{Olsson}, the induced morphism on the quasi-coherent sheaves $\phi_* : Qcoh( [X/G] ) \rightarrow Qcoh (M)$ is well defined.
\end{proof}
\subsection{On the exactness of the functor of invariants}
\begin{lemme} \label{exactinv} 
The map $\rho : [X/G]\rightarrow Y$ (defined in Remark \ref{quotient}) induces a functor $\rho_* : Qcoh([X/G]) \rightarrow Qcoh(Y)$. 
Then, the functor $\rho_*$ is exact if and only if the functor of invariants 
$(-)^A : {}_B\mathcal{M}^A \rightarrow {}_C\mathcal{M}$ is exact.
\end{lemme}
\begin{proof}
By \cite[Lemma IV.10]{these}, we know that the map $\rho$ is quasi-compact and quasi-separated, thus by \cite[Lem. 6.5(i)]{Olsson}, the induced morphism on the quasi-coherent sheaves $\rho_* : Qcoh( [X/G] ) \rightarrow Qcoh (M)$ is well defined.
$\rho$ is defined as making the diagram bellow commute: 
$$
\xymatrix{ [X/G] \ar[rr]^\rho & &Y  \\ &X \ar[lu]^p \ar[ru]_\pi &}
$$
Thus, it induces the following commutating diagram of functors:
$$
\xymatrix{ Qcoh ([X/G]) \ar[rr]^{\rho_*} & & Qcoh (Y)  \\ &Qcoh^G(X) \ar[lu]^{p_*} \ar[ru]_{\pi_*} &} 
$$ 
The functor $p_*$ is an equivalence of categories (see \cite[Proposition IV.24]{these}). 
Moreover, by \cite[\S 5]{Hartshorne}, we have the following commutative diagram:  
$$
\xymatrix{ Qcoh ^G ( X) \ar[r]^{\pi_*} \ar[d]_-{\Gamma (X, \_)  } & \ar[d]^-{\Gamma (Y, \_) } Qcoh (Y) \\ {}_B \mathcal{M}^A  \ar[r]_{(\_)^A} &{}_C\mathcal{M} }
$$ 
where the functors of global section $\Gamma (X, \_) $ and $\Gamma ( Y, \_)$  are equivalences of categories. 
This proves the lemma.
\end{proof}
\begin{theoreme}\label{corocoarse}
Suppose that $X$ is noetherian, finitely presented over $S$, that $G$ is finitely presented over $S$ and that all the inertia groups are finite. If the functor of invariants $(-)^A$ is exact, then the map $\rho : [X/G] \rightarrow Y$ is a coarse moduli space of $[X/G]$ and $\rho$ is proper. \end{theoreme}
\begin{proof}, 
By \cite[Theorem 6.6]{Alper}, the quotient $\pi: X \rightarrow Y$ is categorical in the category of algebraic spaces. \cite[Theorem 1.1]{conrad} insures that the quotient stack admits a coarse moduli space that we denote $\rho : [X/G] \rightarrow M$ such that $\rho$ is proper. But, since $\rho \circ p :  X \rightarrow [X/G] \rightarrow M$ and $\pi : X \rightarrow Y$ are categorical quotients in the category of algebraic spaces, we obtain by unicity that $M \simeq Y$ and that $\rho : [X/G] \rightarrow Y$ is a coarse moduli space. 
\end{proof}
\subsection{Definition of tame quotient stack}
Finally, we can define a tame quotient stack. Since we need the existence of a coarse moduli space. Until the end of section 5, we suppose that $G$ is flat, finitely presented over $S$, that $X$ is finitely presented over $S$ and that all the inertia groups at the geometric points 
for the action $(X,G)$ are finite. Moreover, we denote $M$ the coarse moduli space and $\phi: [X/G] \rightarrow M$ the proper map (see Theorem \ref{coarseexistence}).
\begin{definition} \label{AOVdef}
We say that the quotient stack $[X/G]$ is \textbf{tame} if the functor $\phi_* : Qcoh [X/G] \rightarrow Qcoh (M)$ is exact. 
\end{definition}
\begin{remarque}
\begin{enumerate} 
\item The notion of tameness can be defined similarly for general stacks (see \cite[Definition 3.1]{AOV}).
\item When $G$ is finite flat over $S$, $[X/G]$ is tame if and only if the functor of invariants $(-)^A$ is exact. (It is a consequence of Lemma \ref{exactinv} since $Y$ is a coarse moduli space for $G$ finite flat over $S$, by \cite[\S 3]{conrad}).
\item When $G$ is finite flat over $S$, $G$ is linearly reductive if and only if the classifying stack $B_S G$ is tame.
\end{enumerate}
\end{remarque}
Tameness is local. In fact, 
\begin{lemme} \label{qslocal} \cite[Proposition 3.10]{Alper}
For any morphism of $S$-schemes $g : M'\rightarrow M$, we consider the following $2$-cartesian diagram: 
$$\xymatrix{ [X/G]_{M'} \ar[r]^{g'} \ar[d]_{\phi '} & [X/G] \ar[d]^{\phi } \\ M' \ar[r]_g & M}$$
Suppose that $\phi$ (resp. $\phi '$) is the coarse moduli space for $[X/G]$ (resp. $[X/G]_{M'}$).
\begin{enumerate}
\item If $g$ is faithfully flat and the quotient stack $[X/G]_{M'}$ is tame then the quotient stack $[X/G]$ is tame. 
\item If the quotient stack $[X/G]$ is tame then the quotient stack $[X/G]_{M'}$ is also tame. 
\end{enumerate}
\end{lemme}
\begin{proof}
\begin{enumerate}
\item From the $2$-cartesian diagram we deduce that the functors $g^* \phi_*$ and $\phi '_*g'^*$ are isomorphic. Since $g$ is flat, $g'$ is flat as well and $g'^*$ is exact; also $\phi' _*$ is exact by assumption, so the composite $\phi '_*g'^*$ is exact, hence so is $g^* \phi_*$. But, since $g$ is faithfully flat, we have that $\phi_*$ is also exact as required. 
\item First suppose that $g$ is an open, quasi-compact immersion. Let $$0 \rightarrow \mathcal{F}'_1\rightarrow   \mathcal{F}'_2 \rightarrow \mathcal{F}'_3\rightarrow 0 $$ be an exact sequence of $\mathcal{O}_{[X/G]_{M'}}$-modules. Set $\mathcal{F}_3:= g'_*\mathcal{F}_2 / g'_* \mathcal{F}_1$. Then $$0\rightarrow g'_* \mathcal{F}'_1 \rightarrow g'_* \mathcal{F}'_2 \rightarrow \mathcal{F}_3 \rightarrow 0$$ is exact. Moreover, $g'^* \mathcal{F}_3 \simeq \mathcal{F}'_3$ since the adjunction morphism $g'^*g'_* \rightarrow id$ is an isomorphism. Since $\phi_*$ is exact by assumption, $\phi_*g'_* \mathcal{F}'_2 \rightarrow \phi_* \mathcal{F}_3$ is also surjective, but then $g_* \phi'_* \mathcal{F}'_2 \rightarrow \phi_* \mathcal{F}_3$ is surjective as well. Since $g$ is an open immersion, $\phi'_* \mathcal{F}'_2 \rightarrow g^* \phi_* \mathcal{F}_3$ is surjective. Finally, since $g^* \phi_*$ and $\phi '_* g'_*$ are isomorphic functors, $\phi'_* \mathcal{F}'_2 \rightarrow \phi'_* \mathcal{F}'_3$ is surjective. \\
We consider now any morphism of schemes $g : M' \rightarrow M$. Since the tameness property on a stack is Zariski local, we can assume $M'$ and $M$ affine. Then $g'$ is also affine, so the functor $g'_*$ is exact. By assumption $\phi_*$ is exact, therefore $\phi_* g'_* = g_* \phi '_*$ is exact. But the functor $g_*$ has the property that a sequence $ \mathcal{F}_1\rightarrow   \mathcal{F}_2 \rightarrow \mathcal{F}_3 $ is exact if and only if $ g_* \mathcal{F}_1\rightarrow   g_* \mathcal{F}_2 \rightarrow g_*\mathcal{F}_3 $ is exact. It follows that $\phi '_*$ is exact as required.
\end{enumerate}
\end{proof}
\subsection{Local definition of tameness}
\begin{theoreme}\label{slices}\cite[Theorem 3.2]{AOV}
The following assertions are equivalent: 
\begin{enumerate}
\item The quotient stack $[X/G]$ is tame.
\item The inertia groups $I_G( \xi) \rightarrow Spec(k)$ are linearly reductive groups, for any $\xi : Spec (k) \rightarrow X$, where $k$ is a field.
\item The inertia groups $I_G( \xi) \rightarrow Spec(k)$ are linearly reductive groups, for any geometric point $\xi : Spec (k) \rightarrow X$, where $k$ is an algebraically closed field.
\item The inertia groups $I_G( \mathfrak{p}) \rightarrow Spec(k(\mathfrak{p}))$ are linearly reductive groups, for any $\mathfrak{p} \in X$.
\item For any point $\mathfrak{p}\in X$, denote $\mathfrak{q}$ its image via the morphism $X \rightarrow M$ there exist an fppf (also can be chosen étale surjective) morphism $M' \rightarrow M$ containing $\mathfrak{q}$ in its image, a linearly reductive group scheme $H \rightarrow M'$, such that $H_{k(\mathfrak{q})}\simeq I_G(\mathfrak{p})$, acting on a finite and finitely presented scheme $P \rightarrow M'$ and an isomorphism of algebraic stacks over $M'$ $$[X/G] \times_M M' \simeq [P/H].$$
\end{enumerate}
\end{theoreme}

\begin{proof}
$(1) \Rightarrow (2)$ Let $\xi : Spec(k) \rightarrow X$ be a $k$-point, where $k$ is a field, and $I_G(\xi)$ is the inertia group in $\xi$. The quotient stack $[G_k / I_G(\xi)]$ is a scheme, denote it $G_k / I_G ( \xi) $. Since the square $$ \xymatrix{ G_k /I_G (\xi ) \ar[d] \ar[rr] && X\times_S Spec(k)  \ar[d] \\ B_k I_G (\xi ) \ar[rr]& & [X/G]\times_S Spec(k)}$$ is $2$-cartesian, $B_k I_G (\xi )\rightarrow [X/G]\times_S Spec(k)$ is affine since $G_k/ I_G (\xi )\rightarrow X \times_S Spec(k)$ is affine. \\  
Now, let us consider the following commutative diagram: 
$$
\xymatrix{B_k I_G (\xi ) \ar[d]_{(-)^{I_G (\xi )}} \ar[rr]^g && [X/G] \ar[d]^{\phi} \\ Spec ( k( \mathfrak{q}) ) \ar[rr]_f & & M}  \ \ (\star)
$$
Since we have seen that $g$ is affine, $g_* :Qcoh ( B_k I_G (\xi ) ) \rightarrow Qcoh ( [X/G])$ is an exact functor and $\phi_* : Qcoh ([X/G] ) \rightarrow Qcoh (Y )$ is exact by definition of tameness. Since $f_* \phi'_* = \phi_* g_*$, if $$ 0 \rightarrow V_1 \rightarrow V_2 \rightarrow V_3 \rightarrow 0$$
is an exact sequence of $G$-representations, considered as exact sequence of quasi coherent sheaves over $B_k I_G ( \xi)$, we have the following exact sequence: 
$$
0\rightarrow f_* (V_1)^{I_G (\xi )} \rightarrow f_* (V_2)^{I_G (\xi )} \rightarrow f_* (V_3)^{I_G (\xi )} \rightarrow 0.
$$ Moreover, $(-)^{I_G (\xi )}$ is left exact and this implies that 
$$
0 \rightarrow (V_1)^{I_G (\xi )}\rightarrow (V_2)^{I_G (\xi )}\rightarrow (V_3)^{I_G (\xi )}\rightarrow 0
$$
is exact.
So, $I_G (\xi )$ is linearly reductive. \\
$(2)  \Rightarrow (3)$ Immediate. \\ 
$(3)  \Rightarrow (4)$  By \cite[Theorem 2.16]{AOV}, it is enough to prove that $I_G (\mathfrak{p})_{\Omega}$ is trivial, for any $\Omega$ algebraically closed field. Denoting by ${\mathfrak{p}_{\Omega}}$ the composite $Spec( \Omega) \rightarrow Spec(k(\mathfrak{p})) \rightarrow X$. But,  $I_G (\mathfrak{p})_{\Omega}$ is a subgroup of $I_G( \mathfrak{p}_{\Omega}) $ which is trivial by assumption. Thus, $I_G ( \mathfrak{p})_{\Omega}$ is trivial. 
\\
$(4)  \Rightarrow (5)$ See \cite[Theorem 3.2]{AOV}  \\
$(5) \Rightarrow (1)$ See Lemma \ref{qslocal}.
\end{proof}
We obtain the following corollary. 
\begin{corollaire}  \label{localtamestack2}\cite[Corollary 3.5]{AOV}
The stack $[X/G] \rightarrow S$ is tame if and only if for any morphism $Spec(k) \rightarrow S$, where $k$ is an algebraically closed field, the geometric fiber $[X/G] \times_S Spec(k)\rightarrow Spec(k)$ is tame. \end{corollaire}
\begin{definition}
We say that $[X/G]$ is \textbf{tame at $\mathfrak{p}$} if the inertia group $I_G(\mathfrak{p})$ is linearly reductive.
\end{definition}
\begin{remarque}
By the previous theorem, $[X/G]$ is tame if and only if $[X/G]$ is tame at $\mathfrak{p}$, for any $\mathfrak{p}\in X$.
\end{remarque}
\subsection{Existence of a torsor}
We state now an interesting consequence of the previous theorem which permits to define a torsor for a tame quotient stack for actions of finite commutative group schemes.
\begin{proposition}\label{tamefreeext}
Suppose that $G$ is finite, commutative flat over $S$. 
If the quotient stack $[X/G]$ is tame, then for any point $\mathfrak{p} \in X$, denoting by $\mathfrak{q}\in Y $ its image by the morphism $X \rightarrow Y$, there exist a finitely presented flat morphism $Y' \rightarrow Y$ containing $\mathfrak{q}$ in its image and a subgroup $H$  of $G_{Y'}$ over $Y'$ lifting the inertia group at $\mathfrak{p}$ such that  $X_{Y'}^Y /H$ is a $G_{Y'}/H$-torsor over $Y'$. 
\end{proposition}
\begin{proof}
Notice first that under the assumption of the theorem, $G_{Y'}/H$ is a finite affine group scheme. By the previous theorem, we know that the inertia at $\mathfrak{p}$ is linearly reductive. So, by Lemma \ref{extension}, there are a finitely presented flat morphism $Y' \rightarrow Y$ containing $\mathfrak{q}$ in its image and a linearly reductive group $H\rightarrow Y'$ lifting this inertia group as a subgroup of $G_{Y'}$. 
Moreover, the inertia group at $\overline{\mathfrak{p}}$, image of $\mathfrak{p}$ by the quotient morphism $X_{Y'}^Y \rightarrow X_{Y'}^Y /H$ for the action $(X_{Y'}^Y/H, G_{Y'}/H)$, is equal to $$I_{G_{Y'}/H} ( \overline{\mathfrak{p}})= I_G (\mathfrak{p}) / (I_G( \mathfrak{p}) \cap H_{k(\mathfrak{p})})= \{ e \}$$ By Theorem \ref{locfree}, up to passing to a finitely presented flat neighborhood of $\mathfrak{q}$, the action $(X_{Y'}^Y /H , G_{Y'}/H)$ is free, thus $X_{Y'}^Y /H$ is a $G_{Y'}/H$-torsor over $Y'$ since $G_{Y'}/H$ is finite (see Theorem \ref{locfree}).
\end{proof}
\begin{remarque}
It would be useful to establish the previous proposition in general if $G$ is not commutative. But, in this case $G/H$ is not necessarily a group scheme and we cannot even define the notion of action or of torsor. But, taking the normal closure of $H$ in $G_{Y'}$ instead of $H$, we can establish the same result.
\end{remarque}
\subsection{Slice theorem for tame quotient stacks}
We manage to prove a slice theorem for actions by a finite commutative group scheme, using the following lemma.
\begin{lemme} \cite[Proposition 6.5]{Boas} Let $H$ be a subgroup of $G$ such that the quotient for the natural translation action exists and is universal. Let $\psi : X \rightarrow G/H$ be a morphism of schemes preserving the $G$-actions. Let $$Z = X \times_{G/H} e$$
be defined as the fibered product of the two maps $\psi$ and the inclusion $e=H/H \rightarrow G/H$. Assume that the balanced product $Z\times^H G$ exists and is a universal quotient. Then we have an isomorphism of $G$-actions 
$$(X,G) \simeq (Z\times^H G, G).$$
\end{lemme}
Finally, we get the following slice theorem which extends \cite[Theorem 6.4]{Boas} as we will see above:
\begin{theoreme}\label{commslice}
Suppose that $G$ is commutative and finite over $S$. The quotient stack $[X/G]$ is tame if and only if  the action $(X,G)$ admits finitely presented flat slices such that the slice group at $\mathfrak{p} \in X$ are linearly reductive. 
\end{theoreme}
\begin{proof}
When $G$ is commutative, since by proposition \ref{tamefreeext}, up to passing to a finite finitely presented flat neighborhood of $\mathfrak{q}$, there is a subgroup $H$ of $G$ such that $(X/H , G/H)$ defines a torsor over $Y$, then $X/H \times_Y X/H \simeq X/H \times_{Y} G/H$. This gives us a $G$-equivariant morphism $\psi : X \rightarrow G/H$ after the fppf base change $X/H \rightarrow Y$. So, $(X,G)$ is induced by the action $(Z, H)$ using the notation of the previous lemma. The converse is due to Theorem \ref{slices} since if the action $(X,G)$ admits finitely presented flat slices such that the slice group at $\mathfrak{p} \in X$ are linearly reductive, then for any $\mathfrak{p}\in X$, $I_G(\mathfrak{p})$ is linearly reductive.
\end{proof}
\section{Tame action by affine group scheme vs Tame quotient stack}
\subsection{Algebraic interlude}
The following lemma is quite important in the sense that it relates the functor of invariants to the functor $Com_A(-, B)$. As a consequence, it relates in particular the exactness of this functors which will be important in order to compare the two notions of tameness (see Lemma \ref{corexact}). 
\begin{lemme} \label{cotensisom}
Suppose that the Hopf algebra $A$ is flat over $R$. Let $B \in \mathcal{M}^A$ and $M \in {}_R\mathcal{M}$ be finitely presented as $R$-modules. There is a natural (functorial on $M$) isomorphism:
$$(B \otimes_R M^* )^A \simeq Com_A (M, B)$$
\end{lemme}
\begin{proof}
The commutativity of the right square of the diagram below insures the existence of $\lambda$ which is an isomorphism and so $\alpha$ and $\gamma$ as well. 
$$ \xymatrix{ 0 \ar[r] & (B \otimes_R M^*)^A \ar[d]^-{\lambda} \ar[r] & B \otimes_R M^* \ar[d]^-{\alpha} \ar[r]^-{\phi} & (B\otimes_R M^*)\otimes_R A\ar[d]^-{\gamma} \\ 0 \ar[r] & Com_A ( M, B ) \ar[r] & Hom_R ( M, B) \ar[r]_-{\psi} & Hom_R (M , B \otimes_R A) }$$
where $\phi :=  \rho_{B\otimes_R M^*}  - Id_B \otimes Id_{M^*} \otimes 1$, $\psi:= (\rho_B \circ -)   -  \rho_{M^*}$ and for any $b\otimes f \otimes a \in B \otimes_R M^* \otimes_R A$ and $m \in M$, $\gamma  ( b \otimes f \otimes a) (m)= \sum b. f(m_{(0)}) \otimes a. m_{(-1)}$\\
Indeed, the map $\gamma$ is an isomorphism as composite of the isomorphism $\gamma_1 : B\otimes_R M^* \otimes_R A \rightarrow B \otimes_R A \otimes_R M^*$ defined for any $b\otimes f \otimes a\in B\otimes_R M^* \otimes_R A$ by $\gamma_1 (b\otimes f \otimes a ) =b\otimes a.m_{(-1)}\otimes f(m_{(0)})$ (Its inverse map is defined for any $b \otimes a \otimes f \in B \otimes_R A \otimes_R M^*$ by $\gamma_1^{-1}( b \otimes a \otimes f )= \sum b \otimes f(m_{(0)}) \otimes a. S(m_{(-1)})$) with the canonical isomorphism $\gamma_2 : B\otimes_RA  \otimes_R M^* \rightarrow Hom_R ( M  , B\otimes_R A )$. 
For any $b \otimes f \in B \otimes_R M^*$, we have
$$\begin{array}{lll} \psi \circ \alpha ( b \otimes f) & = & \psi ( [m \mapsto b f(m)])= [ m \mapsto \rho_B (b) f(m) - \sum b f(m_{(0)}) \otimes m_{(-1)}]\end{array} $$ 
and
$$\begin{array}{lll} \gamma \circ \phi ( b\otimes f ) &= & \gamma ([m \mapsto \sum b_{(0)} \otimes f(m_{(0)} ) \otimes b_{(1)} S(m_{(-1)})- b \otimes f(m) \otimes 1]) \\ &=& [m \mapsto (b_{(0)} \otimes b_{(1)} S(m_{(_1)}) (m_{(0)})_{(-1)} ) f((m_{(0)})_{(0)} - \sum b f(m_{(0)} \otimes m_{(-1)}]  \\ &= & [m\mapsto \sum (b_{(0)} \otimes b_{(1)} \epsilon (m_{(-1)} ) )f(m_{(0)} ) -\sum b f(m_{(0)} ) \otimes m_{(-1)}] \\ & = & [m \mapsto \rho_B (b) f(m) - \sum b f(m_{(0)} )\otimes m_{(-1)} ]. \end{array} $$
\end{proof}
\subsection{Tame actions}
 Chinburg, Erez, Pappas and Taylor defined in their article \cite{Boas} the notion of tame action. We recall here this definition and some useful properties.
\begin{definition} 
We say that an action \textbf{$(X, G)$ is tame} if there is a unitary (that is $\alpha (1_A ) = 1_B$)
morphism of $A$-comodules $\alpha : A \rightarrow B$, which means that $\alpha$ is  a $R$-linear map such that $\rho_B \circ \alpha = (\alpha \otimes Id_A) \circ \delta$, such a morphism is called a \textbf{total integral}. \end{definition} 
Tame actions are stable under base change.
\begin{lemme} \label{basechange2} If the action $(X,G)$ is tame, then after an affine base change $R' \rightarrow R$ the action $(X_{R'} , G_{R'})$ is also tame.
\end{lemme}
\begin{proof}
Since the action $(X,G)$ is tame, there is a $A_C$-comodule map $A \rightarrow B$, which induces naturally a comodule map $A_{R'} \rightarrow B_{R'}$. 
\end{proof}

The next lemma will allow us to assume that the base is equal to the quotient, if the structural map $X\rightarrow S$ has the same properties as the quotient morphism $X\rightarrow Y$.
\begin{lemme}\label{Cmodérée}
The following assertions are equivalent:
\begin{enumerate}
\item The action $(X, {}_CG)$ over $C$ is tame .
\item The action $(X,G)$ over $S$ is tame.
\item The action $({}_CX , {}_CG)$ over $C$ is tame.
\end{enumerate}
\end{lemme}
\begin{proof}
$(1)\Rightarrow (2)$ Let $\alpha : {}_C A \rightarrow B$ be a total integral for the tame action $(X,{}_CG)$. Since $B$ is a $ {}_C  A$-comodule via $\rho_B : B \rightarrow B \otimes_R A \simeq B \otimes_C {}_C A$, the composite $\alpha':=\alpha \circ i_A : A \rightarrow B$, where $i_A : A \rightarrow C \otimes_R A$ maps $a$ to $1\otimes a$, is a unitary $A$-comodule map so a total integral for the action $(X, G)$.\\
$(2) \Rightarrow (3)$ Follows from base change, by Lemma \ref{basechange2}. \\ 
$(3) \Rightarrow (1)$ Denote by $\beta :  {}_C A \rightarrow  {}_C  B$ the total integral for the action $({}_CX, {}_CG)$.   Recall that $Id_C \otimes_R A$ is a $ {}_C A$-comodule via $Id_C \otimes \delta$ and $ {}_C  B$ is a $ {}_C A$-comodule map via $Id_C \otimes \rho_B$. Consider the composite $\beta' : = \mu_C \circ \beta : C \otimes_R A \rightarrow B$ where $\mu_C: C \otimes_R B \rightarrow B$ comes from the algebra multiplication of $B$. For any $b\in B$ and $c\in C$, $$\begin{array}{lll}(\mu_C \otimes Id_A )((Id_C \otimes \rho_B) (c \otimes b))&=&(\mu_C \otimes Id_A ) ( \sum c \otimes b_{(0)} \otimes b_{(1)})\\ &=& \sum c b_{(0)} \otimes b_{(1)}= \rho_B (c) \rho_B (b) \text{ (since $C= B^A$)}\\& = & \rho_B (\mu_C(b\otimes c)) \text{ (since $B$ is an $A$-comodule algebra).}\end{array}$$ Thus, $\mu_C$ is an $A_C$-comodule map and $\beta '$ also being compositions of $A_C$-comodule maps. 
\end{proof}
\subsection{Exactness of the functor of the invariants for tame actions}
In the case of a constant group scheme, by \cite[Lemma 2.2]{Boas}, the tameness of an action is equivalent to the surjectivity of the trace map, which generalizes a characterization of tameness in number theory (see \cite[Chapter I, \S 5, Theorem 2]{CasFro}) and justifies the choice of the terminology. 
In the general case, we can define a projector which plays the role of the trace map in the case of an action by a constant group scheme (see \cite[section 1]{Doi}).
\begin{lemme} \cite[\S 1]{Doi}\label{projecteur}
From a total integral map $\alpha : A \rightarrow B$ and $M\in {}_B\mathcal{M}^A$, we can define a $R$-linear projector called \textbf{Reynold operator} 
$$\begin{array}{llll} pr_M: &M &\rightarrow &M^A\\ &m &\mapsto & pr_M(m):=pr_{M, \alpha } (m) = \sum_{(m)}\alpha (S(m_{(1)}))m_{(0)}.\end{array}$$
\end{lemme}
\begin{proof}
For any $m \in M$ and $a \in A$, we have $$\rho_M(\alpha (a))= (\alpha \otimes A) \delta (a) = (\alpha \otimes A) (\sum a_1 \otimes a_2) = \sum \alpha (a_1) \otimes a_2.$$ Thus, $$\rho_M (m. \alpha ( a)) = \sum m_{(0)} \alpha (a_1) \otimes m_{(1)} a_2 \ (\square ).$$
We obtain: $$\begin{array}{lll}  \rho_M (pr_M (m)) &=& \rho_M (\sum m_{(0)} \alpha (S(m_{(1)})) \\ &=& \sum (m_{(0)})_{(0)}\alpha (S(m_{(1)})_1 \otimes (m_{(0)})_1S(m_{(1)})_1)  \ (\text{by } (\square ) )\\
&=& \sum (m_{(0)})_{(0)}\alpha (S((m_{(1)})_2)\otimes (m_{(0)})_1S((m_{(1)})_1) \text{ (since $S$ is an antimorphism) }\\
&=&\sum m_{(0)} \alpha (S(m_{(3)}))\otimes m_{(1)}S(m_{(2)}) = \sum m_{(0)} \alpha (S(m_{(2)}))\otimes (m_{(1)})_1 S(m_{(1)})_2 \\
&=& \sum m_{(0)} S(m_{(2)}) \otimes \epsilon (m_{(1)})1 \text{ (by the definition of an antipode})\\
&=& \sum m_{(0)} S(m_{(1)} \epsilon (m_{(1)})) \otimes 1 = \sum m_{(0)} S((m_{(1)})_1 \epsilon (m_{(1)})_1) \otimes 1 \\
&=& \sum m_{(0)} S(m_{(1)})\otimes 1 \text{(by the properties of the counity).}\\
&=& pr_M (m) \otimes 1 
\end{array} $$
So, $pr_M (m) \in (M)^A $ and ${pr_M}^2=pr_M$. Moreover, for any $m \in (M)^A$, $\rho_M(m)=m\otimes 1$, Hence, $pr_M (m) =m \alpha (s(1))=m\alpha (1)=m$ since $\alpha (1) = 1$. This proves that $pr_M$ is a Reynold operator.\end{proof}
The existence of this projector insures that the functor of invariants in exact, for tame actions. The following lemma will permit to relate the previous two notions of tameness. 
\begin{lemme} \cite[Lemma 2.3]{Boas}\label{inv} If the action $(X, G)$ is tame then the functor of invariants $(-)^A: {}_B\mathcal{M}^A \rightarrow {}_C\mathcal{M}$ is exact. \end{lemme}
\begin{proof}
Let $0\rightarrow M_1 \rightarrow M_2 \rightarrow M_3 \rightarrow 0$ be an exact sequence of ${}_B \! \mathcal{M}^A$. Using the notation of the previous lemma, we have the following commutative diagram:
$$\xymatrix{
&& 0 \ar[d] &\\
0 \ar[r] & (M_1)^A=pr_{M_1}(M_1) \ar[r] \ar[d]_{\phi_1|_{{(M_1)}^A}}& \ar@/_2pc/[l]^{pr_{M_1}} M_1  \ar[d]^{\phi_1} \dar[r]^-*[@]{\hbox to 0pt{\hss\txt{ $\rho_{M_1}$ \\  \\ }\hss}}_-{{M_1}\otimes 1} & M_1\otimes A \ar[d] \\
0 \ar[r] & (M_2)^A=pr_{M_2}(M_2) \ar[r] \ar[d]_{\phi_2|_{{(M_2)}^A}}& \ar@/_2pc/[l]^{pr_{M_2}} M_2  \ar[d]^{\phi_2} \dar[r]^-*[@]{\hbox to 0pt{\hss\txt{ $\rho_{M_2}$ \\  \\ }\hss}}_-{{M_2}\otimes 1} & M_2\otimes A \ar[d] \\
0 \ar[r] & (M_3)^A=pr_{M_3}(M_3) \ar[r] & \ar@/_2pc/[l]^{pr_{M_3}} M_3  \ar[d] \dar[r]^-*[@]{\hbox to 0pt{\hss\txt{ $\rho_{M_3}$ \\  \\ }\hss}}_-{{M_3}\otimes 1} & M_3\otimes A  \\
&&0 &\\
  }$$
Left exactness is automatic, right exactness follows from the previous diagram.
\end{proof}
Adding finite hypothesis, we are able to prove that the exactness of the functor of invariant is equivalent to tameness of the action. 
\begin{proposition} \label{relativemt} \label{corexact}
Suppose that $C$ is locally noetherian, that $B$ is flat over $C$ and that $A$ is finite, locally free over $R$. Then, the following assertions are equivalent:\begin{enumerate}
\item[(1)] The action $(X,G)$ is tame.
\item[(2)] The functor $(-)^A : {}_B \mathcal{M}^A\rightarrow {}_C \mathcal{M}$, $N \mapsto (N)^A$ is exact.
\item[(1')] The action $(X , {}_C G)$ is tame. 
\item[(2')] The functor $(-)^{{}_C A} : {}_B \mathcal{M}^{{}_C A}\rightarrow {}_C \mathcal{M}$, $N \mapsto (N)^{{}_C A}$ is exact.
\end{enumerate}
\end{proposition}
\begin{proof}
$(1) \Leftrightarrow (1')$ follows from lemma \ref{Cmodérée}.\\
$(2) \Leftrightarrow (2')$ follows from ${}_B\mathcal{M}^A= {}_B\mathcal{M}^{A_C}$.\\
$(1)\Rightarrow (2)$ follows from lemma \ref{projecteur}\\ 
$(2') \Rightarrow (1')$ Suppose that $(-)^A$ is exact. Since ${}_C A$ is finite over $C$ (base change of $A$ which is finite over $R$) and $B$ is finite over $C$, $B \otimes_C {}_C A$ is finite over $C$. Moreover, as $C$ is locally noetherian, $B$ and $B \otimes_RA$ are also of finite presentation as algebras over $C$ so, in particular, as $C$-modules. By lemma \ref{cotensisom} since we suppose $B$ flat over $C$, we have the following isomorphism$$(B\otimes_R B^*)^{{}_C A} \simeq Com_{{}_C A}(B, B)  \ and \ (B \otimes_R (B \otimes_RA)^*)^{{}_C A} \simeq Com_{{}_C A} (B \otimes_C {}_C A, B)$$
Since $(B\otimes \epsilon)$ is a $C$-linear section of $\rho_B : B \rightarrow B \otimes_R A$, $(B\otimes_C {}_C A)^* \rightarrow B^*$ is surjective. Thus, from the exactness of $(-)^{{}_C A}$, we obtain the surjectivity of $$(B\otimes_R (B \otimes_RA)^*)^{{}_C A} \rightarrow (B \otimes_R B^*)^{{}_C A}$$ Finally, the isomorphisms above imply the surjectivity of the natural map $$Com_{{}_C A} (B \otimes_C {}_C A, B)\rightarrow Com_{{}_C A} (B , B)$$ This insures the existence of a ${}_C A$-comodule map $\lambda_B : B \otimes_C {}_C A \rightarrow B$ such that $\lambda_B  \circ \rho_B  =B$. The previous lemma permits to conclude the proof.
\end{proof}
\subsection{Relationship between the two notions of tameness}
First, we also can prove easily that the tame actions defines always tame quotient stacks.
\begin{theoreme}\label{inertiefini}
Suppose that $X$ is noetherian, finitely presented over $S$, that $G$ is finitely presented over $S$ and and that all the inertia groups are finite. If the action $(X,G)$ is tame then the quotient stack $[X /G]$ is tame. 
\end{theoreme}
\begin{proof}
This follows directly from Lemma \ref{inv} and Corollary \ref{corocoarse}.\\
\end{proof}
We manage to prove the equivalence of these two notions of tameness defined in this  paper for actions involving finite group schemes such that the $B$ is flat over $C$:
\begin{theoreme}\label{proppal}
Suppose that $G$ is finite and locally free over $S$. The action $(X,G)$ is tame if and only if the quotient stack $[X/G]$ is tame. If moreover, $C$ is noetherian and $B$ flat over $C$, the converse is true. \end{theoreme}
\begin{proof}
By \cite[\S 3]{conrad}, we know that $\rho : [X/G]\rightarrow Y$ is a coarse moduli space.  Moreover, $\rho_* : Qcoh ([X/G]) \rightarrow Qcoh(Y)$ is exact by Lemma \ref{inv}, thus $[X/G]$ is tame. The converse follows from Proposition \ref{corexact}..
\end{proof}
\begin{remarque}
We can replace the hypothesis $Y$ noetherian by $X$ of finite type over $S$. In fact, by \cite[Theorem 3.1, (2)]{conrad}, Y is of finite type over $S$ so also noetherian since $S$ is supposed noetherian.
\end{remarque}
We obtain easily the following corollary. 
\begin{corollaire}\label{lintame}
Suppose that $S$ is noetherian and $G$ is finite and flat over $S$. Then, $G$ is linearly reductive over $S$ if and only if the trivial action $(S, G)$ is tame over $S$. 
\end{corollaire}
\begin{remarque}
We say that a Hopf algebra $A$ is \textbf{relatively cosemisimple} if for all $M \in \mathcal{M}^A$ the submodules which are direct summands in $\mathcal{M}_R$ are also direct summands in $\mathcal{M}^A$. The Hopf algebra $A$ is relatively cosemisimple if and only if the trivial action $(Spec(R), G)$ is tame (see \cite[16.10]{wisbauer}). In particular, when $G$ is finite, flat over $S$, this is also equivalent to $G$ being linearly reductive.
\end{remarque}
The following result can be seen as an analogue of the trace surjectivity that we have proved in the constant case. 
\begin{corollaire}\label{projecteureq}
Suppose that $G$ is finite, locally free over $S$, C is locally noetherian and B flat over C. 
The following assertions are equivalent: 
\begin{enumerate}
\item The action $(X,G)$ is tame.
\item The quotient stack $[X/G]$ is tame. 
\item The functor $(-)^A : {}_B \mathcal{M}^A\rightarrow {}_C \mathcal{M}$, $N \mapsto (N)^A$ is exact.
\item There is a Reynold operator $pr_M : M \rightarrow M ^A$ for any $M \in {}_B\mathcal{M}^A$. 
\end{enumerate}
\end{corollaire}
\begin{proof}
$(1) \Leftrightarrow (2) \Leftrightarrow (3)$ Follows from the previous theorem.\\
$(1)\Rightarrow (4)$ follows from Lemma \ref{projecteur}. \\
$(4)\Rightarrow (2)$ As the functor $(-)^A$ is left exact, it is enough to prove exactness on the right. So, let $\xi : M \rightarrow N \in {}_B Bim^A (M,N)$ be an epimorphism. It induces a morphism $\xi^A: M^A \rightarrow N^A$. For $n \in N^A$, by the surjectivity of $\xi$, there is $m \in M$ such that $\xi (m) = n$. Moreover, $pr_M(m)\in M^A$, so $\xi (pr_M (m)) = pr_N (\xi (m))= n $. Therefore $\xi$ is surjective. 
\end{proof}
\appendix
\bibliography{biblio} 
 \end{document}